\documentclass{amsart}

\usepackage{amsmath}
\usepackage{amssymb}
\usepackage{amsthm}
\usepackage{ascmac}
\usepackage{cases}
\usepackage{multicol}
\usepackage[all]{xy}
\usepackage{nccmath}
\usepackage{graphicx}
\usepackage{amscd}
\usepackage[all]{xy}
\usepackage{float}
\usepackage{bm} 
\usepackage{type1cm} 
\usepackage{color}


\newtheorem{theorem}{Theorem}[section]
\newtheorem{proposition}[theorem]{Proposition}
\newtheorem{lemma}[theorem]{Lemma}
\newtheorem{corollary}[theorem]{Corollary}

\theoremstyle{definition}

\newtheorem{example}[theorem]{Example}
\newtheorem{definition}[theorem]{Definition}
\newtheorem*{notation}{Notation}

\newcommand{\type}{\operatorname{type}}
\newcommand{\kernel}{\operatorname{Ker}}
\newcommand{\im}{\operatorname{Im}}
\newcommand{\cl}{\operatorname{cl}}
\newcommand{\flow}{\mathrm{Flow}}
\newcommand{\flowtrivial}{\mathrm{Flow}_\mathrm{trivial}}
\newcommand{\col}{\mathrm{Col}}
\newcommand{\rank}{\operatorname{rank}}

\newcommand{\z}{\mathbb{Z}}

\newcommand{\cut}{\mathrm{cut}}

\begin{document}

\title[The tunnel number and the cutting number with constituent handlebody-knots \, \,]
{The tunnel number and the cutting number with constituent handlebody-knots}

\author{Tomo Murao}
\address[T. Murao]
{Institute of Mathematics, University of Tsukuba, 1-1-1 Tennoudai, Tsukuba, Ibaraki 305-8571, Japan.}
\email{t-murao@math.tsukuba.ac.jp}



\date{}

\maketitle

\begin{abstract}
We give lower bounds for the tunnel number of knots and handlebody-knots.
We also give a lower bound for the cutting number, 
which is a ``dual" notion to the tunnel number in handlebody-knot theory.
We provide necessary conditions to be constituent handlebody-knots 
by using $G$-family of quandles colorings.
The above two evaluations are obtained as the corollaries.
Furthermore, 
we construct handlebody-knots 
with arbitrary tunnel number and cutting number.
\end{abstract}

\section{Introduction}

The tunnel number of a knot $K$ in the 3-sphere $S^3$ is defined to be 
the minimal number of mutually disjoint arcs $\gamma_1,\ldots, \gamma_t$ properly embedded in $E(K)$ 
such that $E(K \cup \gamma_1\cup \cdots \cup \gamma_t )$ becomes a handlebody, 
where $E(\cdot)$ denotes the exterior.
We call the collection of the arcs $\{\gamma_1,\ldots, \gamma_t \}$ an unknotting tunnel system for $K$.
The study of the tunnel number of knots is closely related to 
that of hyperbolic structures, Heegaard splittings and its Goeritz groups and so on of the exterior.
Indeed, 
for a knot $K$, 
each unknotting tunnel system $\{\gamma_1,\ldots, \gamma_t \}$ of $K$ provides 
a genus $t$ Heegaard splitting of $E(K)$, 
and any genus $t$ Heegaard splitting of $E(K)$ 
is obtained in this manner.
In addition, 
many results concerning the additivity of tunnel number of knots under connected sum 
are often obtained through discussions on Heegaard splittings (for example, see \cite[etc.]{Kim13,Morimoto93,Morimoto00,Morimoto15,SS99,YL11}).
Moriah and Rubinstein \cite{MR97} showed that 
an evaluation formula of tunnel numbers is best possible 
by using arguments from hyperbolic geometry.
Cho and McCullough \cite{CM09-1,CM09-2,CM10,CM11} gave an effective method 
for the study of unknotting tunnels of knots with tunnel number $1$ 
through discussions on Goeritz groups.

The definition of the tunnel number of knots 
is extended to that of handlebody-knots in the same way, 
where a handleboy-knot is a handlebody embedded in $S^3$, 
which is a generalization of a knot concerning a genus.
The study of handlebody-knot theory is suitable for that of unknotting tunnel systems 
since the operation of adding a ``tunnel" has a closure property in handlebody-knot theory, 
that is, 
a handlebody-knot and its unknotting tunnel system $\{\gamma_1,\ldots, \gamma_t \}$ 
can be realized as a sequence of $t+1$ handlebody-knots.
Hence we can evaluate the tunnel number step by step 
through arguments from handlebody-knot theory.
Actually, 
Ishii \cite{Ishii08} gave a lower bound for the tunnel number of handlebody-knots 
by using dihedral quandle colorings for handlebody-knots.

We may regard the tunnel number of a handlebody-knot $H$ 
as the minimal number of 2-handles that must be ``removed" from $E(H)$ such that it becomes a handlebody.
In this paper, we introduce a geometric invariant for handlebody-knots, 
called the cutting number, 
which is defined to be the minimal number of 2-handles that must be ``attached" to $E(H)$ such that it becomes a handlebody.
In this sense, 
the tunnel number and the cutting number are ``dual" geometric invariants for handlebody-knots 
which have finite values.
In this paper, 
for a handlebody-knot $H$, 
we define a constituent handlebody-knot of $H$ 
by a handlebody-knot obtained from $H$ 
by removing an open regular neighborhood of some meridian disks of $H$.
By introducing the notion of constituent handlebody-knots, 
we can deal with the tunnel number and the cutting number of handlebody-knots uniformly.

Ishii \cite{Ishii08} introduced an enhanced constituent link of a spatial trivalent graph, 
and Ishii and Iwakiri \cite{II12} introduced an $A$-flow of a spatial graph, 
where $A$ is an abelian group, 
to define colorings and invariants of handlebody-knots.
Ishii, Iwakiri, Jang and Oshiro \cite{IIJO13}
introduced a $G$-family of quandles 
to extend the above structures.
Ishii and Nelson \cite{IN17} introduced a $G$-family of biquandles, 
which is a biquandle version of a $G$-family of quandles.
However, recently 
the author \cite{Murao-pre} proved that 
there is a one-to-one correspondence between 
the set of a $G$-family of biquandles colorings and that of a $G$-family of quandles colorings 
for any handlebody-knot.
Hence, 
in this paper, 
we give necessary conditions to be constituent handlebody-knots 
by using $G$-family of quandles colorings.
We also give lower bounds for the tunnel number, 
which is a generalization of Ishii's result in \cite{Ishii08}, 
and the cutting number of handlebody-knots.

The outline of the paper is as follows. 
In Section 2, 
we introduce constituent handlebody-knots, 
the tunnel number and the cutting number of handebody-knots.
In Section 3, 
we review a coloring for handlebody-knots 
by using a $G$-family of quandles.
In Section 4, 
we consider module structures of coloring sets by $G$-families of Alexander quandles 
and give some examples of such coloring sets.
In Section 5, 
we provide necessary conditions to be constituent handlebody-knots 
by using $G$-family of quandles colorings.
Furthermore, as the corollaries, 
we give lower bounds for the tunnel number and the cutting number of handlebody-knots.
In Section 6, 
we construct a family of handlebody-knots 
which do not contain a certain classical knot as a constituent handlebody-knot.
Moreover, 
we construct handlebody-knots 
with arbitrary tunnel number and cutting number.

\section{Preliminaries}
A \emph{handlebody-link} is the disjoint union of handlebodies embedded in the 3-sphere $S^3$ \cite{Ishii08}.
A \emph{handlebody-knot} is a one component handlebody-link, 
which is a generalization of a knot with respect to a genus.
In this paper, 
we assume that every component of a handlebody-link 
is of genus at least $1$.
An \emph{$S^1$-orientation} of a handlebody-link is an orientation 
of all genus 1 components of the handlebody-link, 
where an orientation of a solid torus is an orientation of its core $S^1$.
Two $S^1$-oriented handlebody-links $H_1$ and $H_2$ are \emph{equivalent}, 
denoted $H_1 \cong H_2$, 
if there exists an orientation-preserving self-homeomorphism of $S^3$ 
sending one to the other 
preserving the $S^1$-orientation.

A \emph{spatial trivalent graph} is a finite trivalent graph embedded in $S^3$.
In this paper, 
a trivalent graph may have a circle component, 
which has no vertices.
A \emph{Y-orientation} of a spatial trivalent graph is 
an orientation of the graph without sources and sinks with respect to the orientation (Figure \ref{Y-orientation}).
A vertex of a Y-oriented spatial trivalent graph can be allocated a sign; 
the vertex is said to have sign $+1$ or $-1$.
The standard convention is shown in Figure \ref{Y-orientation}.
For a Y-oriented spatial trivalent graph $K$ 
and an $S^1$-oriented handlebody-link $H$, 
we say that 
$K$ \emph{represents} $H$ 
if $H$ is a regular neighborhood of $K$ 
and the $S^1$-orientation of $H$ agrees with the Y-orientation.
Any $S^1$-oriented handlebody-link can be represented by 
some Y-oriented spatial trivalent graph.
We define a \emph{diagram} of an $S^1$-oriented handlebody-link 
by a diagram of a Y-oriented spatial trivalent graph 
representing the handlebody-link.
An $S^1$-oriented handlebody-link is \emph{trivial}
if it has a diagram with no crossings.
In particular, 
$H$ is an $S^1$-oriented trivial handlebody-knot 
if and only if 
the exterior is a handlebody.
In this paper, 
we denote by $O_g$ the $S^1$-oriented genus $g$ trivial handlebody-knot.

\begin{figure}[htb]
\begin{center}
\includegraphics[width=45mm]{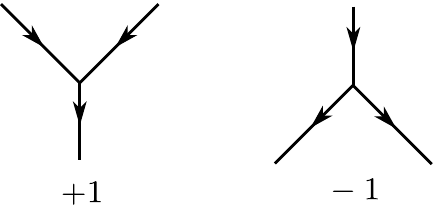}
\end{center}
\caption{Y-orientations and signs of a vertex.}\label{Y-orientation}
\end{figure}

Then the following theorem holds.

\begin{theorem}[\cite{Ishii15-2}]\label{Reidemeister moves}
For any diagrams $D_1$ and $D_2$ of $S^1$-oriented handlebody-links $H_1$ and $H_2$ respectively, 
$H_1$ and $H_2$ are equivalent 
if and only if 
$D_1$ and $D_2$ are related 
by a finite sequence of R1--R6 moves 
depicted in Figure \ref{Reidemeister move} 
preserving Y-orientations.
\end{theorem}

\begin{figure}[htb]
\begin{center}
\includegraphics[width=125mm]{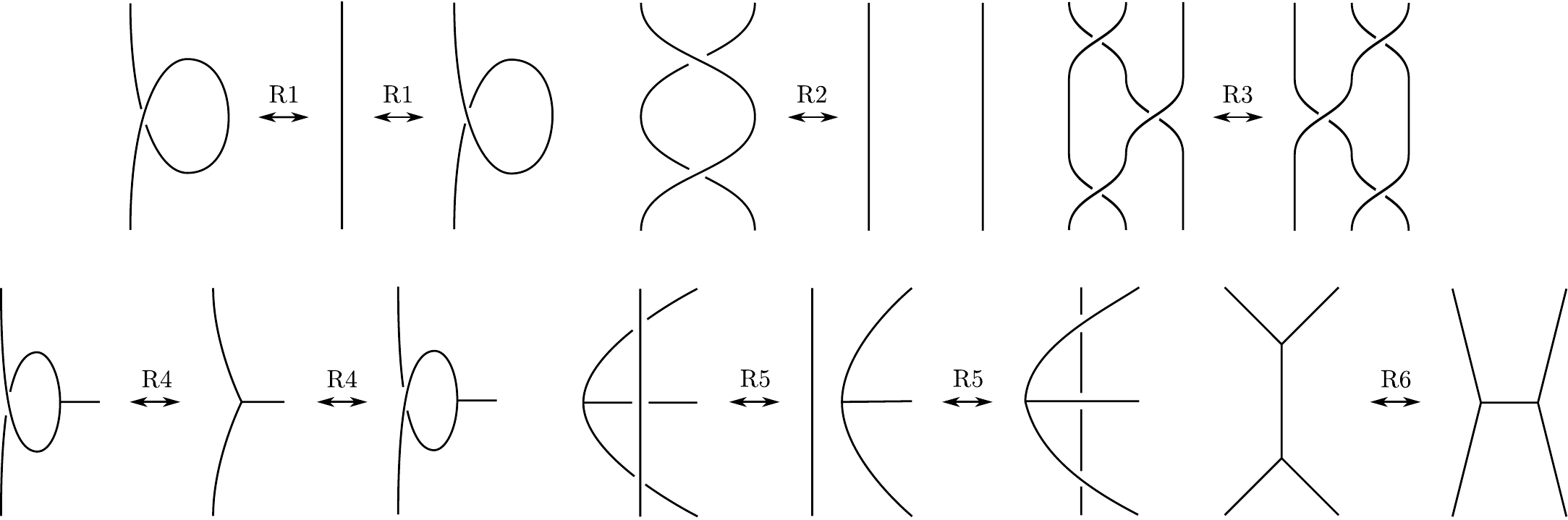}
\caption{The Reidemeister moves for handlebody-links.}\label{Reidemeister move}
\end{center}
\end{figure}

\begin{notation}
Throughout the paper, 
for any diagram $D$ of an $S^1$-oriented handlebody-link, 
we denote by $\mathcal{A}(D)$, $C(D)$ and $V(D)$ 
the set of all arcs, crossings and vertices of $D$ respectively.
An orientation of an arc of $D$ 
is also represented by the normal orientation 
obtained by rotating the usual orientation counterclockwise 
by $\pi/2$ on the diagram.
For any $m \in \mathbb{Z}_{\geq 0}$,
we put $\mathbb{Z}_m:=\mathbb{Z}/m\mathbb{Z}$.
For any set $X$, 
we denote by $\#X$ or $|X|$ the cardinality of $X$.
\end{notation}

Let $H$ and $H'$ be genus $g$ and $g'$ $(g'<g)$ handlebody-knots respectively.
We call $H'$ a \emph{constituent handlebody-knot} of $H$, denoted $H'<H$, 
if there exist mutually disjoint meridian disks $\Delta_1, \ldots, \Delta_{g-g'}$ of $H$ 
such that $\cl(H-\bigcup_{i=1}^{g-g'}N(\Delta_i)) \cong H'$, 
where $N(\cdot)$ and $\cl(\cdot)$ denote 
a regular neighborhood and the closure respectively.
For a genus $g$ handlebody-knot $H$, 
a set of mutually disjoint meridian disks $\{\Delta_1, \ldots ,\Delta_l\}$ of $H$ 
is called a \emph{cutting  system} of $H$ 
if $\cl(H-\bigcup_{i=1}^lN(\Delta_i))$ 
is a handlebody standardly embedded in $S^3$, 
which means that the exterior is a handlebody.
We note that 
the genus of the handlebody may be $0$.
Then we define the \emph{cutting number} $\cut(H)$ of $H$ 
by the minimal number of the cardinalities of cutting systems of $H$.
We note that $\cut(O_g)=0$ for any $g$. 
That is, 
\begin{align*}
\cut(H):=
\begin{cases}
\min \{ \# \Theta \mid \Theta : \text{a cutting  system of $H$} \} & (H \ncong O_g),\\
0 & (H \cong O_g).
\end{cases}
\end{align*}

By the definition, the following hold.
\begin{itemize}
\item
$
\cut(H)=
\begin{cases}
\min \{ i \mid O_{g-i}<H \} & (\text{$O_{g} < H$ for some $g$}),\\
g & (\text{$O_{g} \not< H$ for any $g$}).
\end{cases}
$
\item
$0 \leq \cut(H) \leq g$．
\item
$t(H)=\min\{ i \mid H<O_{g+i} \}$,
\end{itemize}
where $t(H)$ is the tunnel number of $H$.
The tunnel number of a handlebody-knot $H$, 
which is a well-known geometric invariant for classical knots, 
is defined to be the minimal number of mutually disjoint arcs $\gamma_1,\ldots, \gamma_t$ properly embedded in $E(H)$ 
such that $E(H \cup \gamma_1\cup \cdots \cup \gamma_t )$ becomes a handlebody, 
where $E(\cdot)$ denotes the exterior.
In other words, 
the tunnel number is the minimal number of 2-handles 
that must be removed from the exterior such that it becomes a handlebody.
On the other hand, 
the cutting number of a handlebody-knot is the minimal number of 2-handles 
that must be attached to the exterior such that it becomes a handlebody.
In this sense, 
we can consider the cutting number of a handlebody-knot 
as a dual notion to the tunnel number.

\section{Colorings by a $G$-family of quandles}

In this section, 
we introduce a $G$-flow and a coloring for $S^1$-oriented handlebody-links 
by using a $G$-family of quandles.

A quandle \cite{Joyce82, Matveev82} is a non-empty set $X$ with a binary operation $* : X \times X \to X$ 
satisfying the following axioms.
\begin{itemize}
\item
For any $x \in X$, 
$x*x=x$.
\item
For any $y \in X$, 
the map $S_y : X \to X$ defined by $S_y(x)=x*y$ is a bijection.
\item
For any $x,y,z \in X$, 
$(x*y)*z=(x*z)*(y*z)$.
\end{itemize}
We define the \emph{type} of a quandle $X$, denoted $\type X$, 
by the minimal number of $n \in \z_{>0}$ 
satisfying $a*^nb=a$ for any $a,b \in X$, 
where for any $i \in \mathbb{Z}$ and $x,y \in X$, 
we define $x*^i y =S_y^i(x)$.
We set $\type X := \infty$ 
if we do not have such a positive integer $n$.
Any finite quandle is of finite type.
Let $X$ be an $R[t^{\pm1}]$-module, 
where $R$ is a commutative ring.
For any $a,b \in X$, we define $a * b=ta+(1-t)b$.
Then $(X,*)$ is a quandle, called an \emph{Alexander quandle}.

Next, we recall the definition of a $G$-family of quandles.

\begin{definition}[\cite{IIJO13}]
Let $G$ be a group with the identity element $e$.
A $G$-family of quandles is a non-empty set $X$ 
with a family of binary operations $*^g:X \times X \to X ~(g \in G)$ 
satisfying the following axioms.
\begin{itemize}
\item
For any $x \in X$ and $g \in G$, 
$x*^gx=x.$
\item
For any $x,y \in X$ and $g,h \in G$, 
$x*^{gh}y=(x*^gy)*^hy$ and $x*^ey=x$.
\item
For any $x,y,z \in X$ and $g,h \in G$, 
$(x*^gy)*^hz=(x*^hz)*^{h^{-1}gh}(y*^hz)$.
\end{itemize}
\end{definition}

Let $R$ be a ring and $G$ be a group with the identity element $e$.
Let $X$ be a right $R[G]$-module, where $R[G]$ is the group ring of $G$ over $R$.
Then $(X,\{ *^g \}_{g \in G})$ is a $G$-family of quandles, called a \emph{$G$-family of Alexander quandles}, 
with $x*^gy=xg+y(e-g)$ \cite{IIJO13}.
Let $(X,*)$ be a quandle 
and put $k:=\type X$.
Then $(X,\{*^i\}_{i \in \mathbb{Z}_{k}})$ is a $\mathbb{Z}_{k}$-family of quandles.
In particular, 
when $X$ is an Alexander quandle, 
$(X,\{*^i\}_{i \in \mathbb{Z}_{k}})$ is called a \emph{$\mathbb{Z}_{k}$-family of Alexander quandles} 
(see Example \ref{ex Alex1}).

Let $D$ be a diagram of an $S^1$-oriented handlebody-link $H$.
It is known that the fundamental group $\pi_1(S^3-H)$ is represented by 
the arcs, crossings and vertices of $D$ as follows.
For a crossing $c$ and a vertex $\tau$ of $D$, 
we denote by $r_c$ the relation $v_c^{-1}u_cv_c=w_c$ and by $r_\tau$ the relation $\alpha_\tau \beta_\tau=\gamma_\tau$, 
where we denote by $u_c, v_c, w_c, \alpha_\tau, \beta_\tau$ and $\gamma_\tau$ 
the arcs incident to $c$ or $\tau$ as shown in Figure \ref{wirtinger}.
The fundamental group $\pi_1(S^3-H)$ is generated by 
the arcs $x$ for each $x \in \mathcal{A}(D)$ and has the relations $r_c$ and $r_\tau$ for each $c \in C(D)$ and $\tau \in V(D)$, 
that is, a presentation of $\pi_1(S^3-H)$ is given by 
\[
\langle x~ (x \in \mathcal{A}(D)) \mid r_c, r_\tau~ (c \in C(D), \tau \in V(D)) \rangle.
\]
We call it the \emph{Wirtinger presentation} of $\pi_1(S^3-H)$ with respect to $D$.

\begin{figure}[htb]
\begin{center}
\includegraphics[width=75mm]{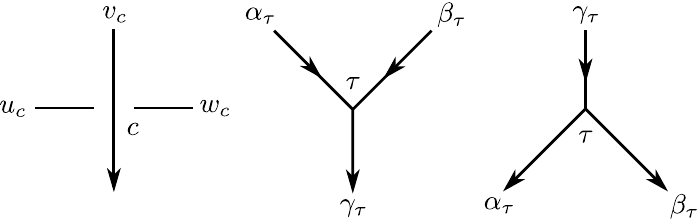}
\end{center}
\caption{Arcs incident to a crossing or a vertex.}\label{wirtinger}
\end{figure}

Let $G$ be a group 
and let $D$ be a diagram of an $S^1$-oriented handlebody-link $H$.
A \emph{$G$-flow} of $D$ is a map $\phi : \mathcal{A}(D) \to G$ satisfying the conditions 
depicted in Figure \ref{flow} 
at each crossing and vertex.
In this paper, 
to avoid confusion, 
we often represent an element of $G$ with an underline.
We denote by $(D,\phi)$, 
called a \emph{$G$-flowed diagram} of $H$, 
a diagram $D$ given a $G$-flow $\phi$ 
and by $\flow(D;G)$ the set of all $G$-flows of $D$.
We can identify a $G$-flow $\phi$ with a group representation of the fundamental group $\pi_1(S^3-H)$ to $G$, 
which is a group homomorphism from $\pi_1(S^3-H)$ to $G$.
Let $D'$ be a diagram of $H$ obtained by 
applying one of Reidemeister moves to $D$ once.
For any $G$-flow $\phi$ of $D$, 
there is a unique $G$-flow $\phi'$ of $D'$ 
which coincides with $\phi$ 
except near the point where the move applied.
This gives a one-to-one correspondence between $\flow(D;G)$ and $\flow(D';G)$.
Since the two $G$-flows $\phi$ and $\phi'$ represent the same group representation $\rho$, 
called a $G$-flow of $H$, 
we often use the symbol $\rho$ instead of $\phi$ and $\phi'$ 
and write $\flow(H;G)$ instead of $\flow(D;G)$ and $\flow(D';G)$.

\begin{figure}[htb]
\begin{center}
\includegraphics[width=75mm]{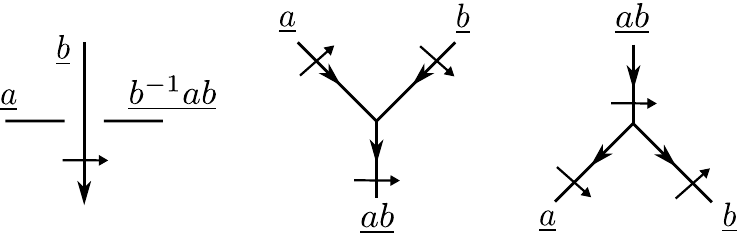}
\end{center}
\caption{A $G$-flow of $D$.}\label{flow}
\end{figure}

Let $X$ be a $G$-family of quandles 
and let $(D,\rho)$ be a $G$-flowed diagram of an $S^1$-oriented handlebody-link.
An \emph{$X$-coloring} of $(D,\rho)$ is a map 
$C : \mathcal{A}(D,\rho) \to X$ satisfying 
the conditions 
depicted in Figure \ref{G-family_of_qnd_col.} 
at each crossing and vertex.
An $X$-coloring $C$ is \emph{trivial} 
if $C$ is a constant map.
We denote by $\col_X(D,\rho)$ 
the set of all $X$-colorings of $(D,\rho)$.
It is easy to see that $\#\col_X(D,\rho) \geq \#X$.

\begin{figure}[htb]
\begin{center}
\includegraphics[width=85mm]{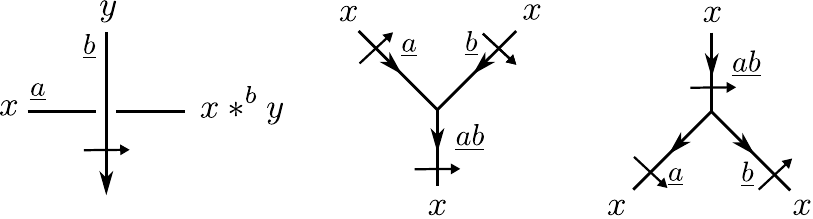}
\end{center}
\caption{A coloring of $(D,\rho)$ by a $G$-family of quandles.}\label{G-family_of_qnd_col.}
\end{figure}

\begin{proposition}[\cite{IIJO13}]
Let $X$ be a $G$-family of quandles, 
$D$ and $D'$ be diagrams of an $S^1$-oriented handlebody-link $H$ 
and let $\rho$ be a $G$-flow of $H$.
Then there is a one-to-one correspondence between 
$\col_X(D,\rho)$ and $\col_X(D',\rho)$.
\end{proposition}

By this proposition, 
for any diagram $D$ of an $S^1$-oriented handlebody-link $H$, 
the multiset $\{ \# \col_X(D,\rho) \mid \rho \in \flow(H;G) \}$ is an invariant of $H$.

For example, 
let $(D,\rho)$ be the $\mathbb{Z}_2$-flowed diagram of the handlebody-knot depicted in Figure \ref{example coloring} 
and 
let $R_3$ be the dihedral quandle, 
that is, $R_3=\z_3$ and $x*y=2y-x$ for any $x,y \in R_3$.
We note that $\type R_3=2$. 
Then $(R_3, \{ *^i \}_{i \in \z_2})$ is a $\z_2$-family of quandles.
Therefore the assignment of elements of $R_3$ to each arc of $(D,\rho)$ 
as shown in Figure \ref{example coloring} is an $(R_3, \{ *^i \}_{i \in \z_2})$-coloring of $(D,\rho)$.

\begin{figure}[htb]
\begin{center}
\includegraphics[width=70mm]{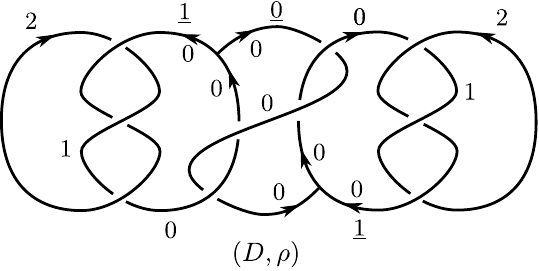}
\end{center}
\caption{A coloring of $(D,\rho)$ by the $\z_2$-family of quandles $(R_3, \{ *^i \}_{i \in \z_2})$.}\label{example coloring}
\end{figure}

Let $D$ be a diagram of an $S^1$-oriented handlebody-link $H$.
A $G$-flow $\rho$ of $H$ is a \emph{trivial coloring $G$-flow} 
if for any $G$-family of quandles $X$ and $C \in \col_X(D,\rho)$, $C$ is a trivial $X$-coloring.
We denote by $\flowtrivial(H;G)$ the set of all trivial coloring $G$-flows of $H$. 

For any group $G$ and $S^1$-oriented handlebody-knot $H$, 
the constant map $\rho_e : \pi_1(S^3-H) \to G$ sending into the identity element $e$ 
is a trivial coloring $G$-flow of $H$ 
since 
for any $G$-family of quandles $X$ and $x,y \in X$, 
it follows that  $x*^ey=x$.

At last in this section, 
we prove the following lemma 
we use in Section 5.

\begin{lemma}\label{trivial hdbdy-knot flow}
For any group $G$, 
every $G$-flow of $O_g$ is a trivial coloring $G$-flow.
\end{lemma}

\begin{proof}
Let $O_g$ be the diagram of the handlebody-knot $O_g$ depicted in Figure \ref{trivial_hdbdy-knot_flow}, 
where we note that we use the same symbol $O_g$ as the genus $g$ trivial handlebody-knot.
Any $G$-flow $\rho$ of $O_g$ is represented as in Figure \ref{trivial_hdbdy-knot_flow}, 
where $a_i \in G$ for any $i=1,\ldots, g$ and $e$ is the identity element of $G$.
Hence it is easy to see that 
for any $G$-family of quandles $X$, 
every $X$-coloring of $(O_g,\rho)$ is trivial.
\end{proof}

\begin{figure}[htb]
\begin{center}
\includegraphics[width=70mm]{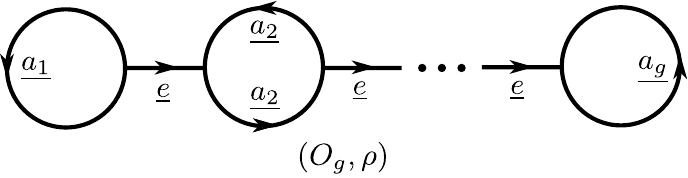}
\end{center}
\caption{A $G$-flow of $O_g$.}\label{trivial_hdbdy-knot_flow}
\end{figure}

By Lemma \ref{trivial hdbdy-knot flow}, 
for any $G$-family of quandles $X$ and $\rho \in \flow(O_g;G)$, 
we obtain that $\#\col_X(O_g,\rho)=\#X$.

\section{Module structures of coloring sets by $G$-families of Alexander quandles}

Let $(D,\rho)$ be a $G$-flowed diagram of an $S^1$-oriented handlebody-link 
and let $X$ be a $G$-family of Alexander quandles as a right $R[G]$-module for some ring $R$.
Then $\col_X(D,\rho)$ is a right $R$-module 
with the action $(C \cdot r) (x) = C(x)r$ 
and the addition $(C + C')(x)= C(x)+C'(x)$ 
for any $C, C' \in \col_X(D,\rho)$, $x \in \mathcal{A}(D,\rho)$ and $r \in R$.
In this section, we consider the module structures of coloring sets by $G$-families of Alexander quandles.

Let $R$ and $R'$ be rings.
We denote by $M(m,n;R)$ the set of $m \times n$ matrices over $R$ 
and set $M(n;R):=M(n,n;R)$.
We denote by $GL(n;R)$ the set of $n \times n$ invertible matrices over $R$.
We can regard a matrix in $M(m,n;M(k,l;R))$ as a matrix in $M(km,ln;R)$.
We call it a \emph{flat matrix}.
For any $(a_{i,j}) \in M(m,n;R)$ and map $f :R \to R'$, 
we define $f((a_{i,j}))=(f(a_{i,j})) \in M(m,n;R')$.

Let $R$ be a commutative ring, $G$ be a group and let $X$ be a right $R[G]$-module.
Then $X$ is also an $R$-module.
We assume that $X$ is a finitely generated free $R$-module, 
that is, 
$X$ is isomorphic to $R^d$ for some $d \in \z_{\geq0}$.
Let $A=(a_{i,j}) \in M(n,m;R[G])$ 
and let $f_A : X^n \to X^m$ be an $R$-homomorphism defined by $f_A((x_1,\ldots, x_n))=(x_1,\ldots, x_n)A$, 
where $(x_1,\ldots, x_n)A$ means $(\sum_{i=1}^{n}x_i a_{i,1}, \ldots, \sum_{i=1}^{n}x_i a_{i,m})$.


An action of $G$ on $X$ is a group homomorphism $\eta:G \to \mathrm{Aut}_{R\mathchar`-\mathsf{Mod}}(X) \cong GL(d;R)$, 
where $R\mathchar`-\mathsf{Mod}$ is the category of $R$-modules, 
and $\mathrm{Aut}_{R\mathchar`-\mathsf{Mod}}(X)$ is the automorphism group of $X$.
Then $\eta$ induces an $R$-homomorphism $\widetilde{\eta} : R[G] \to M(d;R)$ 
satisfying the commutative diagram 
\[
\xymatrix{
G \ar[r]^-\eta \ar@{^{(}-_>}[d]_-{\mathrm{inclusion}} & \mathrm{Aut}_{R\mathchar`-\mathsf{Mod}}(X) \cong GL(d;R) \ar@{^{(}-_>}[rr]^-{\mathrm{inclusion}} & & M(d;R) \\ 
R[G]\ar[rrru]_-{\widetilde{\eta}} & & }.
\]
That is, 
for any $(r_1, \ldots, r_d) \in R^d \cong X$ and $\sum_{g \in G}r_gg \in R[G]$, 
\[
(r_1, \ldots, r_d) \cdot \sum_{g \in G}r_gg
=(r_1, \ldots, r_d) \sum_{g \in G}r_g \eta(g)
=(r_1, \ldots, r_d)  \widetilde{\eta}(\sum_{g \in G}r_gg).
\]
Then for any $A \in M(n,m,R[G])$, 
it follows that 
\begin{align*}
\kernel f_A 
& = \left\{ (x_1,\ldots,x_n) \in X^{n} \middle | (x_1, \ldots, x_{n})A =\bm{0} \right\}\\
& \cong \left\{((r_{1,1},\ldots, r_{1,d}), \ldots, (r_{n,1},\ldots, r_{n,d})) \in (R^d)^{n} \middle | \right. \\
& \left. \qquad \quad((r_{1,1},\ldots, r_{1,d}), \ldots, (r_{n,1},\ldots, r_{n,d})) \widetilde{\eta}(A) =\bm{0} \right\}\\
& \cong \left\{ (r_{1,1}, \ldots, r_{n,d}) \in R^{dn} \middle | (r_{1,1}, \ldots, r_{n,d}) \widetilde{\eta}(A) =\bm{0} \right\}, 
\end{align*}
where $\widetilde{\eta}(A) \in M(n,m;M(d;R))$, 
and we regard $\widetilde{\eta}(A)$ as the flat matrix in $M(dn,dm;R)$ in the last line.
Therefore when $R$ is a field $F$, 
it follows that 
$\kernel f_A$ is a vector subspace of $X^n$ over $F$, 
and $\dim_F \kernel f_A=dn-\rank \widetilde{\eta}(A)$.
In particular, 
if $X$ is an extension field of $F$, 
the map $f_A$ is also an $X$-linear map, 
and $\kernel f_A$ is a vector subspace of $X^n$ over $X$.
An action of $G$ on $X$ is a group homomorphism $\zeta:G \to  \mathrm{Aut}_{X\mathchar`-\mathsf{Vect}}(X) \cong X$, 
where $X\mathchar`-\mathsf{Vect}$ is the category of vector spaces over $X$.
Then $\zeta$ induces an $F$-homomorphism $\widetilde{\zeta} : F[G] \to X$ 
satisfying the commutative diagram 
\[
\xymatrix{
G \ar[r]^-\zeta \ar@{^{(}-_>}[d]_-{\mathrm{inclusion}} & \mathrm{Aut}_{X\mathchar`-\mathsf{Vect}}(X) \cong X \\ 
F[G]\ar[ru]_-{\widetilde{\zeta}} & }.
\]
That is, 
for any $x \in X$ and $\sum_{g \in G}k_gg \in F[G]$, 
\[
x \cdot \sum_{g \in G}k_gg
=x \sum_{g \in G}k_g \zeta(g)
=x  \widetilde{\zeta}(\sum_{g \in G}k_gg).
\]
Then for any $A \in M(n,m,F[G])$, 
it follows that 
\begin{align*}
\kernel f_A 
& = \left\{ (x_1,\ldots,x_n)\in X^{n} \middle | (x_1, \ldots, x_{n})A =\bm{0} \right\}\\
& \cong \left\{ (x_1,\ldots,x_n) \in X^{n} \middle | (x_1, \ldots, x_{n})\widetilde{\zeta}(A) =\bm{0} \right\},
\end{align*}
where $\widetilde{\zeta}(A) \in M(n,m;X)$.
Therefore it follows that 
$\dim_X \kernel f_A=n-\rank \widetilde{\zeta}(A)$ and $d \cdot \dim_X \kernel f_A= \dim_F \kernel f_A$.

In this paper, 
we assume that every component of a diagram of any $S^1$-oriented handlebody-link 
has a crossing at least 1.
Let $(D,\rho)$ be a $G$-flowed diagram of an $S^1$-oriented handlebody-link 
and let $X$ be a $G$-family of Alexander quandles as a right $R[G]$-module for some ring $R$.
We put $C(D,\rho) = \{c_1, \ldots, c_{n_1}\}$ and $V(D,\rho) = \{\tau_1, \ldots, \tau_{2n_2} \}$, 
where $C(D,\rho)$ and $V(D,\rho)$ are the set of all crossings of $(D,\rho)$ and the one of all vertices of $(D, \rho)$ respectively, 
and the sign of $\tau_i$ is 1 for any $i=1, \ldots, n_2$ and $-1$ for any $i=n_2+1, \ldots, 2n_2$.
Put $n:=n_1+3n_2$.
We denote by $x_i$ each arc of $(D, \rho)$ as shown in Figure \ref{arcs}, 
which implies $\mathcal{A}(D,\rho)= \{x_1, \ldots, x_{n} \}$.
We denote by $u_i, v_i, w_i, \alpha_i, \beta_i$ and $\gamma_i$ the arcs incident to a crossing $c_i$ or a vertex $\tau_i$ 
as shown in Figure \ref{notation}.

\begin{figure}[htb]
\begin{center}
\includegraphics[width=80mm]{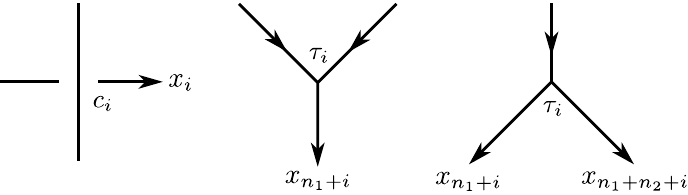}
\end{center}
\caption{Arcs.}\label{arcs}
\end{figure}

\begin{figure}[htb]
\begin{center}
\includegraphics[width=75mm]{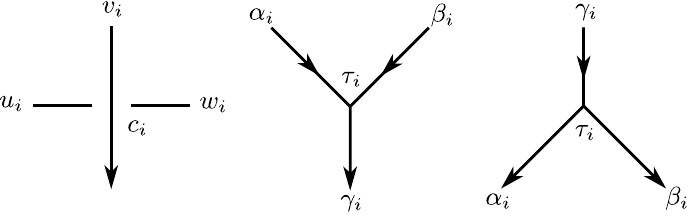}
\end{center}
\caption{Notations.}\label{notation}
\end{figure}

For any arcs $x,x' \in \mathcal{A}(D,\rho)$, 
we put 
\begin{align*}
\delta (x,x'):=
\begin{cases}
1 & (x=x'),\\
0 & (x \neq x').
\end{cases}
\end{align*}
Then we define a matrix 
$A(D,\rho;X)= (a_{i,j}) \in M(n_1+4n_2,n;R[G])$ by 
\begin{align*}
a_{i,j}=
\begin{cases}
\delta(u_i,x_j)\rho (v_i)+\delta(v_i,x_j)(e-\rho (v_i))-\delta(w_i,x_j) & (1 \leq i \leq n_1),\\
\delta(\alpha_{i-n_1},x_j)-\delta(\gamma_{i-n_1},x_j) & (n_1+1 \leq i \leq n_1+2n_2),\\
\delta(\beta_{i-n_1-2n_2},x_j)-\delta(\gamma_{i-n_1-2n_2},x_j) & (n_1+2n_2+1 \leq i \leq n_1+4n_2).
\end{cases}
\end{align*}
We note that 
$A(D,\rho;X)$ is determined 
up to permuting of rows and columns of the matrix.
Then we can identify $\col_X(D,\rho)$ with the right $R$-module 
\begin{align*}
\left\{ (z_1, \ldots, z_{n}) \in X^{n} \middle | (z_1, \ldots, z_{n})A(D,\rho;X)^T=\bm{0} \right\}
\end{align*}
with the action $(z_1, \ldots, z_{n})r=(z_1r, \ldots, z_{n}r)$ 
for any $(z_1, \ldots, z_{n}) \in \col_X(D,\rho)$ and $r \in R$, 
where $A(D,\rho;X)^T$ is the transposed matrix of $A(D,\rho;X)$.
Hence if $R$ is a commutative ring and $X \cong R^d$ as $R$-modules for some $d \in \z_{\geq 0}$, 
it follows that $\col_X(D,\rho) \cong \kernel f_{A(D,\rho;X)^T}$, 
where we remind that $f_{A(D,\rho;X)^T}:X^n \to X^{n_1+4n_2}$ is an $R$-homomorphism 
defined by $f_{A(D,\rho;X)^T}(z_1, \ldots, z_{n})=(z_1, \ldots, z_{n})A(D,\rho;X)^T$.

For example, let $(E,\psi)$ be the $G$-flowed diagram of the handlebody-knot 
depicted in Figure \ref{ex mtx}.
Then for a $G$-family of Alexander quandles $X$ as a right $R[G]$-module, we have 
\begin{align*}
A(E,\psi;X)=
\begin{pmatrix}
b & 0 & e-b & 0 & -1\\
e-a & a & -1 & 0 & 0 \\
0 & 1 & -1 & 0 & 0\\
-1 & 0 & 0 & 1 & 0\\
0 & 0 & -1 & 1 & 0\\
-1 & 0 & 0 & 0 & 1
\end{pmatrix} 
\in M(6,5;R[G])
\end{align*}
and
\begin{align*}
\col_X(E,\psi) \cong 
\left\{ (z_1, \ldots, z_5) \in X^{5} \middle | (z_1, \ldots, z_5)A(E,\psi;X)^T=\bm{0} \right\}.
\end{align*}

\begin{figure}[htb]
\begin{center}
\includegraphics[width=45mm]{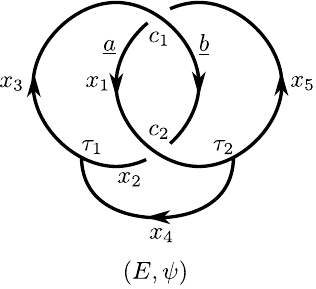}
\end{center}
\caption{A $G$-flowed diagram $(E,\psi)$.}\label{ex mtx} 
\end{figure}

\begin{example}\label{ex Alex1}
Let $X$ be an Alexander quandle as an $R[t^{\pm 1}]$-module for some commutative ring $R$ 
and put $k:= \type X$.
Then $X$ is an $R[\z_k]$-module with $x \cdot t^i=xt^i$ 
for any $x \in X$ and $t^i \in \z_k$, 
where we regard $\z_k$ as $\langle t \mid t^k \rangle$.
Hence $X$ is a $\z_k$-family of Alexander quandles.
Therefore for a $\z_k$-flowed diagram $(D,\rho)$ of an $S^1$-oriented handlebody-link, 
$\col_X(D,\rho)$ is an $R$-module.
When $R$ is a field $F$ and $X \cong F^d$ as vector spaces over $F$ for some $d \in \z_{\geq 0}$, 
it follows that $\col_X(D,\rho)$ is a vector space over $F$, 
and $\dim_F \col_X(D,\rho)=dn-\rank \widetilde{\eta}(A(D,\rho;X))$, 
where $n=\#\mathcal{A}(D,\rho)$.
In particular, if $X$ is an extension field of $F$, 
it follows that $\col_X(D,\rho)$ is also a vector space over $X$, 
and $\dim_X \col_X(D,\rho)=n-\rank \widetilde{\zeta}(A(D,\rho;X))$.
\end{example}

\begin{example}\label{ex Alex2}
Let $R$ be a ring, $X=R^d$ and $G=GL(d;R)$ for some $d \in \z_{\geq 0}$.
Then $X$ is a right $R[G]$-module with 
$(r_1,\ldots, r_d) \cdot (a_{i,j})=(\sum_{i=1}^dr_ia_{i,1}, \ldots, \sum_{i=1}^dr_ia_{i,d})$ 
for any $(r_1,\ldots, r_d) \in X$ and $(a_{i,j}) \in G$.
Hence $X$ is a $G$-family of Alexander quandles.
Therefore for a $G$-flowed diagram $(D,\rho)$ of an $S^1$-oriented handlebody-link, 
$\col_X(D,\rho)$ is a right $R$-module.
When $R$ is a field $F$, 
it follows that $\col_X(D,\rho)$ is a vector space over $F$, 
and $\dim_F \col_X(D,\rho)=dn-\rank \widetilde{\eta}(A(D,\rho;X))$,
where $n=\#\mathcal{A}(D,\rho)$.
\end{example}

%
%

\section{Results}

In this section, 
we provide essential conditions to be constituent handlebody-knots 
by using colorings by $G$-families of quandles.
Furthermore, as the corollaries, 
we give lower bounds for the tunnel number and the cutting number of handlebody-knots.

\begin{theorem}\label{thm col}
Let $H$ and $H'$ be $S^1$-oriented genus $g$ and $g'$ $(g'<g)$ handlebody-knots 
and $D$ and $D'$ be their diagrams respectively.
Let $\rho' \in \flow(H',G)$ and 
$X$ be a $G$-family of Alexander quandles as a right $F[G]$-module for some field $F$, 
where $X \cong F^d$ as vector spaces over $F$ for some $d \in \z_{\geq 0}$.
If $H'<H$, 
there exists $\rho \in \flow(H;G)$ 
such that $\im \rho = \im \rho'$
and
\[
\dim_F \col_X(D',\rho')-\dim_F \col_X(D,\rho) \leq d(g-g').
\]
\end{theorem}

\begin{proof}
Assume that $H'<H$ and put $m:=g-g'$.
There exist $S^1$-oriented handlebody-knots $H_0, H_1, \ldots, H_m$ 
such that $H_0=H'$, $H_m=H$, $H_i<H_{i+1}$ for any $i=0,1,\ldots, m-1$, and the genus of $H_i$ is $g'+i$ for any $i=0,1,\ldots, m$.
For any $\rho_i \in \flow(H_i;G)$, 
the handlebody-knots $H_i$ and $H_{i+1}$ respectively have $G$-flowed diagrams $(D_i,\rho_i)$ and $(D_{i+1},\rho_{i+1})$ 
which are identical except in the neighborhood of a point where they differ as shown in Figure \ref{tunnel}.
Here we may assume that  
the two arcs of $(D_i,\rho_i)$ in the left of Figure \ref{tunnel} are $x_1$ and $x_2$ $(x_1 \neq x_2)$, 
where we put $\mathcal{A}(D_i,\rho_i)=\{x_1,\ldots, x_n\}$.
It is easy to see that $\im \rho_i=\im \rho_{i+1}$.
Then we have 
\begin{align*}
\col_X (D_i,\rho_i) \cong \left\{ (z_1, \ldots, z_n) \in X^n \middle | (z_1, \ldots, z_n) A(D_i,\rho_i;X)^T=\bm{0} \right\}
\end{align*}
as vector spaces over $F$.
Since the coloring set $\col_X(D_{i+1},\rho_{i+1})$ is obtained from $\col_X(D_i,\rho_i)$ 
by adding one relation $z_1=z_2$, we have
\begin{align*}
\col_X(D_{i+1},\rho_{i+1}) & \cong \left\{(z_1, \ldots, z_n) \in X^n\middle |(z_1, \ldots, z_n) A(D_i,\rho_i;X)^T=\bm{0},z_1=z_2\right\}\\
& \cong \left\{ (z_1, \ldots, z_n) \in X^n\middle |(z_1, \ldots, z_n)
\begin{pmatrix}
A(D_i,\rho_i;X)\\
\bm{a}
\end{pmatrix}^T
=\bm{0}\right\}
\end{align*}
as vector spaces over $F$, 
where $\bm{a}=(e,-e,0,\ldots,0)$.
We note that 
$\widetilde{\eta}(e)$ is the $d \times d$ identity matrix.
Therefore it follows that 
\[
0 \leq \rank \widetilde{\eta}(
\begin{pmatrix}
A(D_i,\rho_i;X)\\
\bm{a}
\end{pmatrix}^T) 
-\rank \widetilde{\eta}(A(D_i,\rho_i;X)^T) \leq d
\]
as flat matrices.
Therefore we obtain that 
\[
0 \leq \dim_F\col_X(D_i,\rho_i)-\dim_F\col_X(D_{i+1},\rho_{i+1}) \leq d.
\]

Consequently, 
for any $\rho' \in \flow(H',G)$, 
there exists $\rho \in \flow(H;G)$ 
such that $\im \rho=\im \rho'$ and 
$\dim_F \col_X(D',\rho')-\dim_F \col_X(D,\rho) \leq dm=d(g-g')$.
\end{proof}

\begin{figure}[h]
\begin{center}
\includegraphics[width=70mm]{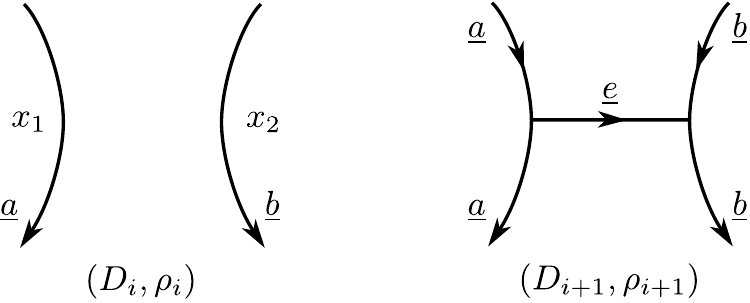}
\end{center}
\caption{Adding an arc.}\label{tunnel}
\end{figure}

\begin{theorem}\label{thm flow}
Let $H$ and $H'$ be $S^1$-oriented genus $g$ and $g'$ $(g'<g)$ handlebody-knots respectively 
and let $G$ be a group.
If $H'<H$, 
it follows that 
\[
\# \flowtrivial (H';G) \leq \# \flowtrivial (H;G).
\]
\end{theorem}

\begin{proof}
Assume that $H'<H$ and put $m:=g-g'$.
There exist $S^1$-oriented handlebody-knots $H_0, H_1, \ldots, H_m$ 
such that $H_0=H'$, $H_m=H$, $H_i<H_{i+1}$ for any $i=0,1,\ldots, m-1$, and the genus of $H_i$ is $g'+i$ for any $i=0,1,\ldots, m$.
For any $\rho_i \in \flowtrivial(H_i;G)$, 
the handlebody-knots $H_i$ and $H_{i+1}$ respectively have 
a trivial coloring $G$-flowed diagram $(D_i,\rho_i)$ and a $G$-flowed diagram $(D_{i+1},\rho_{i+1})$ 
which are identical except in the neighborhood of a point where they differ as shown in Figure \ref{tunnel}.
Assume that $\rho_{i+1}$ is not a trivial coloring $G$-flow of $H_{i+1}$, 
which means that 
there exists a $G$-family of quandles $X$ and 
a non-trivial $X$-coloring $C$ of $(D_{i+1},\rho_{i+1})$.
Then $C$ induces a non-trivial $X$-coloring of $(D_i,\rho_i)$, 
that is, 
the $X$-coloring of $(D_i,\rho_i)$ obtained from $C$ 
by ignoring the arc we added as shown in Figure \ref{tunnel} is not trivial.
This contradicts to $\rho_i \in \flowtrivial(H_i;G)$.
Hence $\rho_{i+1}$ is a trivial coloring $G$-flow of $H_{i+1}$. 
Therefore 
we have a map from $\flowtrivial (H_i;G)$ to $\flowtrivial (H_{i+1};G)$ 
sending $\rho_i$ into $\rho_{i+1}$, 
and it is easy to see that the map is injective.
Consequently, 
we obtain that $\# \flowtrivial (H_i;G) \leq \# \flowtrivial (H_{i+1};G)$, 
which implies that $\# \flowtrivial (H';G) \leq \# \flowtrivial (H;G)$.
\end{proof}

By Theorems \ref{thm col} and \ref{thm flow}, 
we have the following corollaries 
concerning evaluations of the tunnel number and the cutting number 
of handlebody-knots.

\begin{corollary}\label{tunnel number}
Let $H$ be an $S^1$-oriented handlebody-knot 
and $(D,\rho)$ be a $G$-flowed diagram of $H$.
Let $X$ be a $G$-family of Alexander quandles as a right $F[G]$-module for some field $F$, 
where $X \cong F^d$ as vector spaces over $F$ for some $d \in \z_{\geq 0}$.
Then it follows that 
\[
\frac{\dim_F\col_X(D,\rho)}{d}-1 \leq t(H).
\]
\end{corollary}

\begin{proof}
Let $g$ be the genus of $H$ 
and put $m:=t(H)$, which implies that $H<O_{g+m}$.
Let $(O_g,\rho_0)$ be a $G$-flowed diagram of $O_g$, 
where we note that we use the same symbol $O_g$ as the genus $g$ trivial handlebody-knot.
By Lemma \ref{trivial hdbdy-knot flow}, 
we have $\dim_F \col_X(O_g,\rho_0)=d$.
By Theorem \ref{thm col}, 
we obtain that 
$\dim_F \col_X(D,\rho)-d \leq dm$, 
which completes the proof.
\end{proof}

\begin{corollary}\label{cutting number}
Let $H$ be an $S^1$-oriented genus $g$ handlebody-knot 
and let $G$ be a group.
Then it follows that 
\[
g-\log_{|G|}\#\flowtrivial (H;G) \leq \cut(H).
\]
\end{corollary}

\begin{proof}
Put $m:=\cut(H)$ and suppose $m<g$.
Then we have $O_{g-m}<H$.
By Lemma \ref{trivial hdbdy-knot flow}, 
we have $\flowtrivial(O_{g-m};G)=\flow(O_{g-m};G)=|G|^{g-m}$.
Therefore, by Theorem \ref{thm flow}, 
it follows that $|G|^{g-m} \leq \# \flowtrivial (H;G)$, 
which implies that $g-\log_{|G|}\#\flowtrivial (H;G) \leq m$.
When $m=g$, 
we immediately obtain that $g-\log_{|G|}\#\flowtrivial (H;G) \leq m$.
This completes the proof.
\end{proof}

\section{Examples}

In this section, we give some examples.
In Example \ref{ex const hdbdy-knot}, 
we construct a family of handlebody-knots 
which do not contain a certain knot as a constituent handlebody-knot.
In Example \ref{ex tunnel number}, 
we give a family of genus $g$ handlebody-knots with tunnel number $gn$ for any $g \in \z_{>0}$ and $n \in \z_{\geq 0}$.
In Example \ref{ex cutting number}, 
we give a family of genus $g$ handlebody-knots with cutting number $g$ for any $g \in \z_{\geq 2}$.

\begin{example}\label{ex const hdbdy-knot}
Let $K$ and $H_n$ be respectively the knot and the genus 2 handlebody-knot 
represented by the $\z_2$-flowed diagrams $(D,\rho)$ and $(D_n,\rho_n(a,b))$ 
depicted in Figure \ref{ex_const_hdbdy-knot} for any $n \in \z_{\geq 0}$ and $a,b \in \z_2$.
We note that
$K$ is the knot $8_{18}$ in Rolfsen's knot table \cite{Rolfsen76}, 
and $H_1$ is the genus 2 handlebody-knot $5_4$ in the table given in \cite{IKMS12}.
Let $X$ be the Alexander quandle $\z_3[t^{\pm 1}]/(t+1)$, 
which is isomorphic to the field $\z_3$.
Since $\type X=2$, 
$X$ is the $\z_2$-family of Alexander quandles.
Then for any $z_1,z_2,z_3 \in X$, 
the assignment of them to each arc of $(D,\rho)$ as shown in Figure \ref{ex_const_hdbdy-knot} is an $X$-coloring of $(D,\rho)$, 
which implies that $\dim_{\z_3} \col_X(D,\rho) \geq 3$.
On the other hand, 
we can easily see that 
for any $n \in \z_{\geq 0}$ and $a,b \in \z_2$, 
all $X$-colorings of $(D_n,\rho_n(a,b))$ are trivial, 
which implies that 
$\dim_{\z_3} \col_X(D_n,\rho_n(a,b))=1$.
Hence we have $\dim_{\z_3} \col_X(D,\rho)-\dim_{\z_3} \col_X(D_n,\rho_n(a,b)) \geq 2$
for any $n \in \z_{\geq 0}$ and $a,b \in \z_2$.
Therefore 
$K$ is not a constituent handlebody-knot of $H_n$ for any $n \in \z_{\geq 0}$ 
by Theorem \ref{thm col}.
\end{example}

\begin{figure}[h]
\begin{center}
\includegraphics[width=110mm]{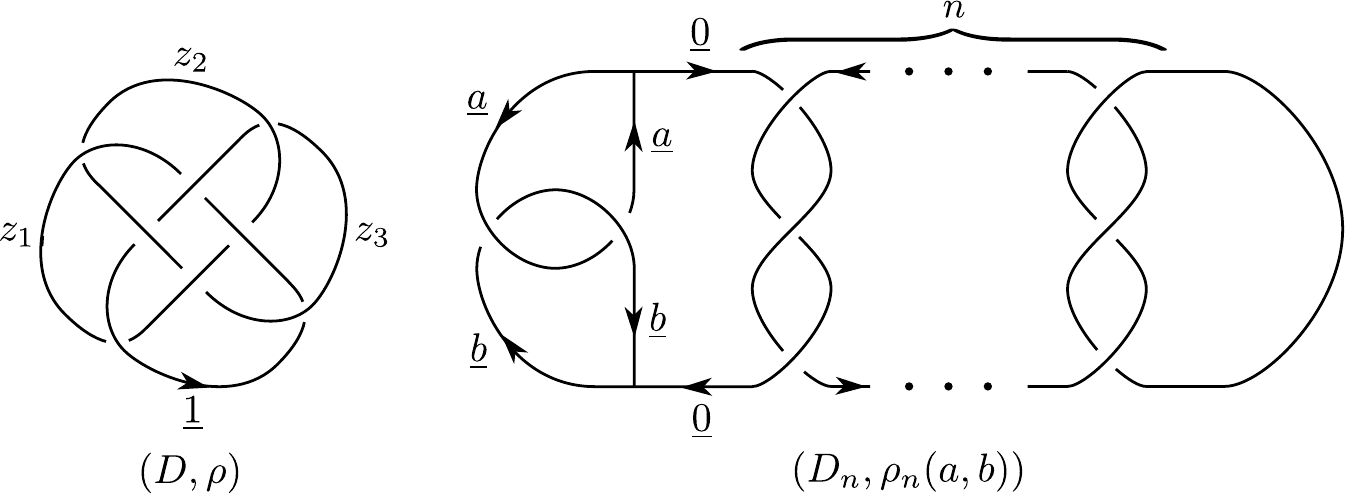}
\end{center}
\caption{$\z_2$-flowed diagrams $(D,\rho)$ and $(D_n,\rho_n(a,b))$ of $K$ and $H_n$ respectively.}\label{ex_const_hdbdy-knot}
\end{figure}

\begin{example}\label{ex tunnel number}
Let $H_{g,n}$ be the $S^1$-oriented genus $g$ handlebody-knot 
represented by the $\z_3$-flowed diagram $(D_{g,n},\rho_{g,n})$ depicted in Figure \ref{ex_tunnel_number} 
for any $g \in \z_{> 0}$ and $n \in \z_{\geq 0}$.
Let $X$ be the Alexander quandle $\z_2[t^{\pm 1}]/(t^2+t+1)$, 
which is an extension field of $\z_2$ 
and isomorphic to $(\z_2)^2$ as vector spaces over $\z_2$.
Since $\type X=3$, 
$X$ is the $\z_3$-family of Alexander quandles.
Then for any $z_0,z_{i,j} \in X$ $(1\leq i \leq g, 1 \leq j \leq n)$, 
the assignment of them to each arc of $(D_{g,n},\rho_{g,n})$ as shown in Figure \ref{ex_tunnel_number} is an $X$-coloring of $(D_{g,n},\rho_{g,n})$, 
which implies that 
\[
\dim_{\z_2} \col_X(D_{g,n},\rho_{g,n})=2 \dim_X \col_X(D_{g,n},\rho_{g,n}) \geq 2(gn+1).
\]
Hence it follows that $gn \leq t(H_{g,n})$ by Corollary \ref{tunnel number}.
On the other hand, 
the set of $gn$ arcs drawn by a dotted line in Figure \ref{ex_tunnel_number} is an unknotting tunnel system for $H_{g,n}$.
Therefore we obtain that $t(H_{g,n})=gn$ for any $g \in \z_{> 0}$ and $n \in \z_{\geq 0}$.
\end{example}

\begin{figure}[h]
\begin{center}
\includegraphics[width=125mm]{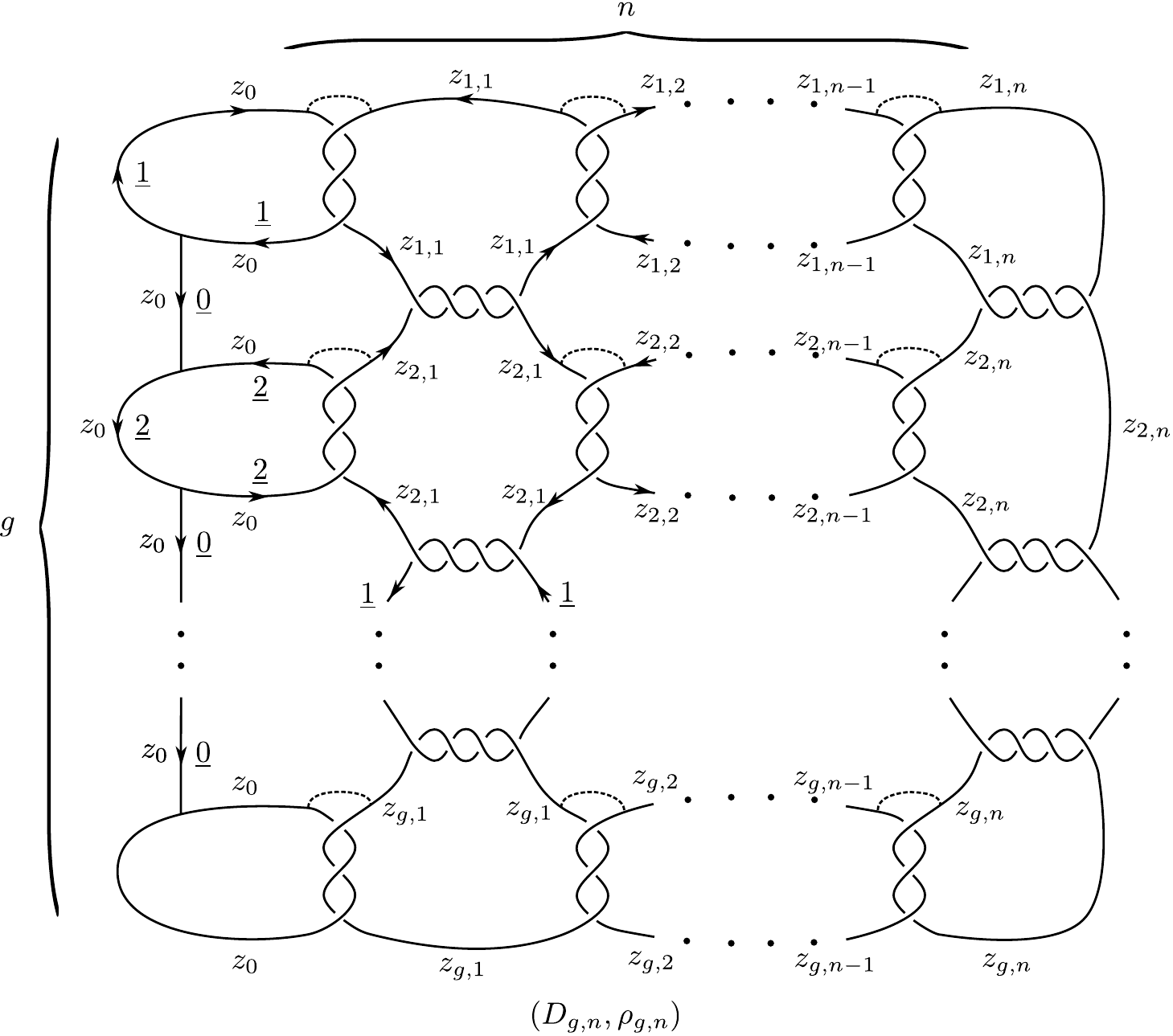}
\end{center}
\caption{A $\z_3$-flowed diagram $(D_{g,n},\rho_{g,n})$ of $H_{g,n}$.}\label{ex_tunnel_number}
\end{figure}

\begin{example}\label{ex cutting number}
For any $g \geq 2$ and $l_1,\ldots, l_g \in 2\z$, 
let $H_{l_1,\ldots,l_g}$ be the genus $g$ handlebody-knot 
represented by the spatial graph $\Gamma_{l_1,\ldots,l_g}$, which is a graph embedded in $S^3$, 
with a $g$-valent vertex $v_g$ depicted in Figure \ref{ex_spatial_graph}, 
which means that $H_{l_1,\ldots,l_g}$ is a regular neighborhood of $\Gamma_{l_1,\ldots,l_g}$.
$H_{l_1,\ldots,l_g}$ has the $\z_2$-flowed diagram $(D_{l_1,\ldots,l_g},\rho(a_1,\ldots,a_g))$ depicted in Figure \ref{ex_cutting_number} 
for any $a_i \in \z_2$.
We note that $\flow(H_{l_1,\ldots,l_g};\z_2)=\{ \rho(a_1,\ldots,a_g) \mid a_i \in \z_2 \}$.
Let $X$ be the Alexander quandle $\z_3[t^{\pm 1}]/(t+1)$, 
which is isomorphic to the field $\z_3$.
Since $\type X=2$, 
$X$ is the $\z_2$-family of Alexander quandles.
Suppose that $(a_1,\ldots,a_g) \neq (0,\ldots,0)$.
Since $\Gamma_{l_1,\ldots,l_g}$ has a $g$-fold rotational symmetry to $v_g$, 
we may assume that $a_1=1$.
First,
if $l_1=4l$ for some $l \in \z$, 
we have the non-trivial $X$-coloring of $(D_{l_1,\ldots,l_g},\rho(a_1,\ldots,a_g))$ 
depicted in the top of Figure \ref{ex_cutting_number_col}.
Next, 
if $l_1=4l+2$ for some $l \in \z$ and $a_2=0$, 
we have the non-trivial $X$-coloring of $(D_{l_1,\ldots,l_g},\rho(a_1,\ldots,a_g))$ 
depicted in the middle of Figure \ref{ex_cutting_number_col}.
Finally,
if $l_1=4l+2$ for some $l \in \z$ and $a_2=1$, 
we have the non-trivial $X$-coloring of $(D_{l_1,\ldots,l_g},\rho(a_1,\ldots,a_g))$ 
depicted in the bottom of Figure \ref{ex_cutting_number_col}.
Hence we have $\flowtrivial(H_{l_1,\ldots,l_g};\z_2)=\{ \rho(0,\ldots,0) \}$, 
that is, $\# \flowtrivial(H_{l_1,\ldots,l_g};\z_2)=1$.
Therefore, 
we obtain that $g \leq \cut(H_{l_1,\ldots,l_g})$, 
which implies that $\cut(H_{l_1,\ldots,l_g})=g$ 
for any $g \geq 2$ and $l_1,\ldots, l_g \in 2\z$ 
by Corollary \ref{cutting number}.
\end{example}

\begin{figure}[h]
\begin{center}
\includegraphics[width=110mm]{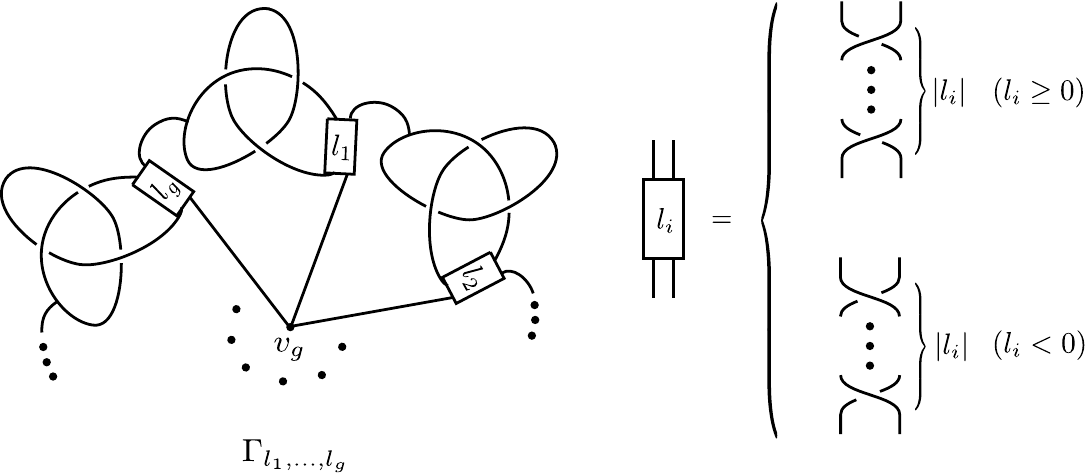}
\end{center}
\caption{A spatial graph $\Gamma_{l_1,\ldots,l_g}$.}\label{ex_spatial_graph}
\end{figure}

\begin{figure}[h]
\begin{center}
\includegraphics[width=120mm]{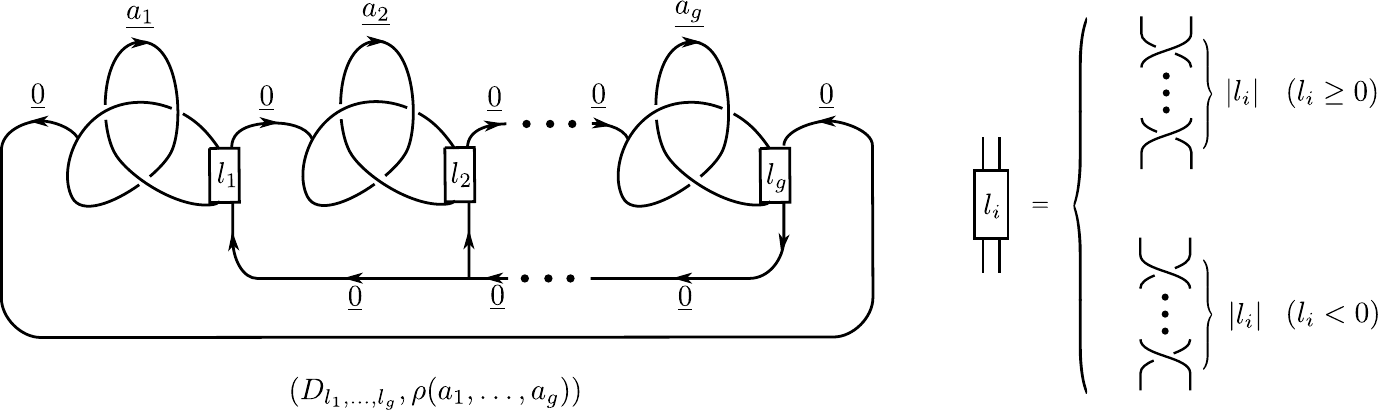}
\end{center}
\caption{A $\z_2$-flowed diagram $(D_{l_1,\ldots,l_g},\rho(a_1,\ldots,a_g))$ of $H_{l_1,\ldots,l_g}$.}\label{ex_cutting_number}
\end{figure}

\begin{figure}[h]
\begin{center}
\includegraphics[width=120mm]{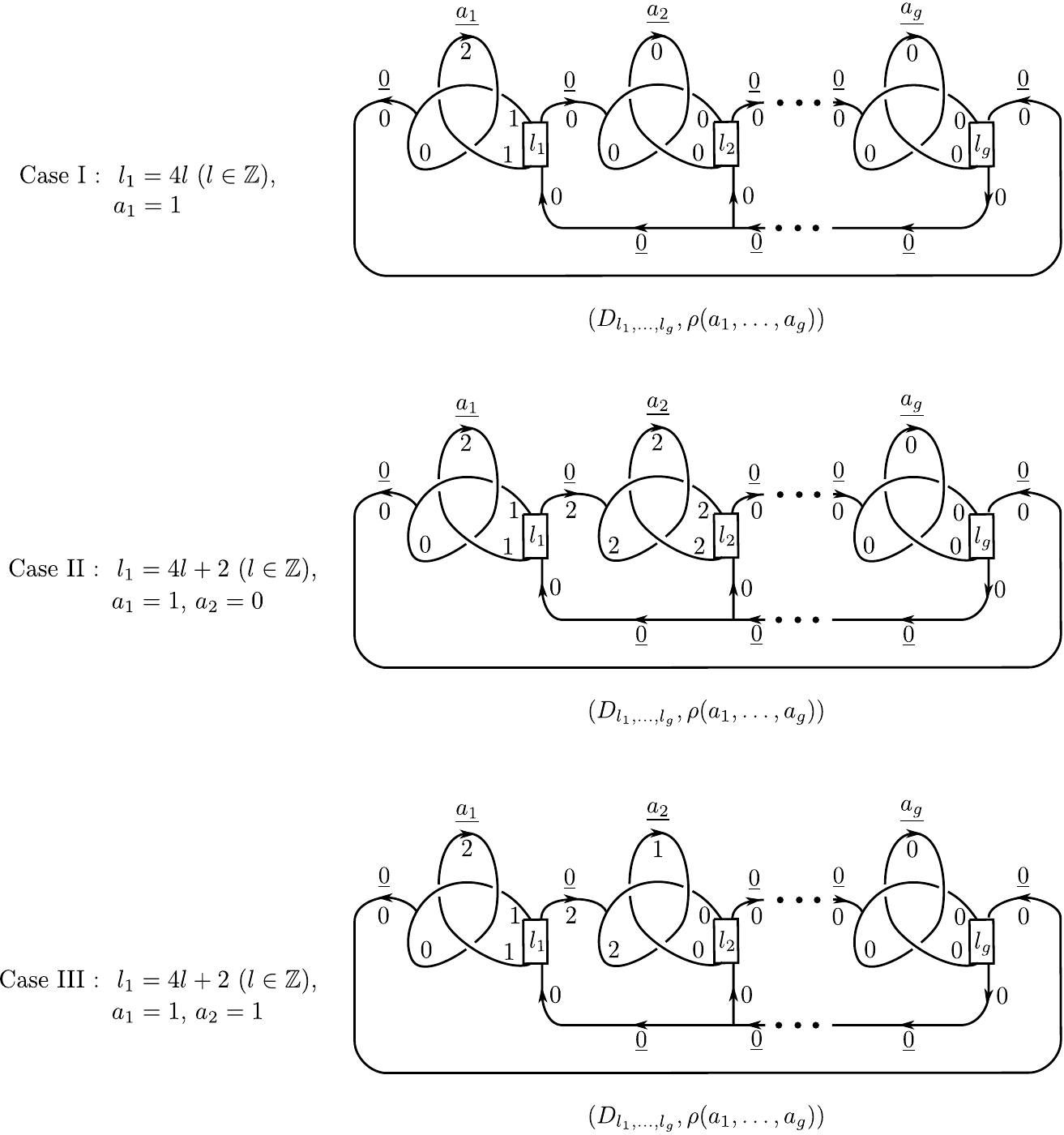}
\end{center}
\caption{Non-trivial $X$-colorings of $(D_{l_1,\ldots,l_g},\rho(a_1,\ldots,a_g))$.}\label{ex_cutting_number_col}
\end{figure}

\section*{Acknowledgment}
The author would like to express his best gratitude to Atsushi Ishii and Shin'ya Okazaki 
for their helpful advice and valuable discussions.



\begin{thebibliography}{99}

\bibitem{CM09-1}
S. Cho and D. McCullough, 
\textit{Cabling sequences of tunnels of torus knots}, 
Algebr. Geom. Topol. \textbf{9}(2009), 1--20.

\bibitem{CM09-2}
S. Cho and D. McCullough, 
\textit{The tree of knot tunnels}, 
Geom. Topol. \textbf{13}(2009), 769--815.

\bibitem{CM10}
S. Cho and D. McCullough, 
\textit{Constructing knot tunnels using giant steps}, 
Proc. Amer. Math. Soc. \textbf{138}(2010), 375--384.

\bibitem{CM11}
S. Cho and D. McCullough, 
\textit{Tunnel leveling, depth, and bridge numbers}, 
Trans. Amer. Math. Soc. \textbf{363}(2011), 259--280.

\bibitem{Ishii08}
A. Ishii,
\textit{Moves and invariants for knotted handlebodies},
Algebr. Geom. Topol. \textbf{8}(2008), 1403--1418.

\bibitem{Ishii15-2}
A. Ishii, 
\textit{The Markov theorems for spatial graphs and handlebody-knots with Y-orientations},
Internat. J. Math. \textbf{26}(2015), 1550116, 23 pp. 

\bibitem{II12}
A. Ishii and M. Iwakiri, 
\textit{Quandle cocycle invariants for spatial graphs and knotted handlebodies},
Canad. J. Math. \textbf{64}(2012), 102--122. 

\bibitem{IIJO13}
A. Ishii, M. Iwakiri, Y. Jang, K. Oshiro, 
\textit{A $G$-family of quandles and handlebody-knots}, 
Ill. J. Math. \textbf{57}(2013), 817--838.

\bibitem{IKMS12}
A. Ishii, K. Kishimoto, H. Moriuchi and M. Suzuki, 
\textit{A table of genus two handlebody- knots up to six crossings}, 
J. Knot Theory Ramifications \textbf{21}(2012), 1250035, 1--9.

\bibitem{IN17}
A. Ishii and S. Nelson, 
\textit{Partially multiplicative biquandles and handlebody-knots}, 
to appear in Contemporary Mathematics.

\bibitem{Joyce82}
D. Joyce,
\textit{A classifying invariant of knots, the knot quandle},
J. Pure Appl. Alg. \textbf{23}(1982), 37--65.

\bibitem{Kim13}
J. Kim, 
\textit{On critical Heegaard splittings of tunnel number two composite knot exteriors}, 
J. Knot Theory Ramifications \textbf{22}(2013), 1350065, 11 pp.

\bibitem{Matveev82}
S. V. Matveev,
\textit{Distributive groupoids in knot theory},
Mt. Sb. (N.S.)  \textbf{119(161)}(1982), 78--88.

\bibitem{MR97}
Y. Moriah and H. Rubinstein, 
\textit{Heegaard structures of negatively curved 3-manifolds},
Comm. in Anal. and Geom. \textbf{5}(1997), 375--412.

\bibitem{Morimoto93}
K. Morimoto, 
\textit{On the additivity of tunnel number of knots},
Topology Appl. \textbf{53}(1993), 37--66.

\bibitem{Morimoto00}
K. Morimoto, 
\textit{On the super additivity of tunnel number of knots},
Math. Ann. \textbf{317}(2000), 489--508.

\bibitem{Morimoto15}
K. Morimoto, 
\textit{On Heegaard splittings of knot exteriors with tunnel number degenerations},
Topology Appl. \textbf{196}(2015), 719--728.

\bibitem{Murao-pre}
T. Murao, 
\textit{A relationship between multiple conjugation quandle/biquandle colorings},
preprint.

\bibitem{Rolfsen76}
D. Rolfsen, 
\textit{Knots and Links}, 
Math. Lecture Series no. 7, Publish or Perish Inc., Berkeley, 1976.

\bibitem{SS99}
M. Scharlemann and J. Schultens, 
\textit{The tunnel number of the sum of $n$ knots is at least $n$}, 
Topology \textbf{38}(1999), 265--270.

\bibitem{YL11}
G. Yang and F. Lei, 
\textit{Some sufficient conditions for tunnel numbers of connected sum of two knots not to go down}, 
Acta Math. Sin. (Engl. Ser.) \textbf{27}(2011), 2229--2244.

\end{thebibliography}
\end{document}